\documentclass[a4paper,reqno]{amsart}
\usepackage{amsmath}
\usepackage{amsfonts}
\usepackage{amssymb}

\newtheorem{theorem}{Theorem}[section]
\newtheorem{lemma}[theorem]{Lemma}
\newtheorem{definition}[theorem]{Definition}

\newtheorem{corollary}[theorem]{Corollary}

\newtheorem{conjecture}[theorem]{Conjecture}

\newcommand{\FF}{\mathbb{F}}
\newcommand{\EE}{\mathbb{E}}
\newcommand{\PP}{\mathbb{P}}
\newcommand{\F}{\mathbb{F}_{p}^{d}}
\newcommand{\E}{\widehat{E}(\xi)}

\newcommand{\m}{|}

\newcommand{\R}{\mathcal{R}}
\newcommand{\I}{\mathcal{I}}
\newcommand{\e}{e^{-\frac{2\pi i x\cdot\xi}{p}}}

\begin{document}

\title{Salem sets in vector spaces over finite fields}

\author{Changhao Chen}

\address{School of Mathematics and Statistics, The University of New South Wales, Sydney NSW 2052, Australia }
\email{changhao.chenm@gmail.com}




\begin{abstract}
We prove that almost all random subsets of a finite vector space are weak Salem sets (small Fourier coefficient), which extends a result of Hayes to a different probability model. 
\end{abstract}


\maketitle

\section{Introduction}

Let $F_p$ denote the finite field with $p$ element where $p$ is prime, and $\F$ be the $d$-dimensional vector space over this field. Let $E \subset \F$. We use the same notation as in Babai \cite{Babai}, Hayes \cite{Hayes} to define that 
\begin{equation}
\Phi(E)=\max_{\xi \neq 0}|\E|.
\end{equation}
Here and in what follows, we  simply write 
$E(x)$  for the
characteristic function of $E$, $\widehat{E}$ it's discrete Fourier transform which we will define it in Section \ref{sec:pre}. For $\xi\neq 0$, we mean that $\xi$ is a non-zero vector of $\F$. Applying the Plancherel identity, we have that for any $E\subset \F$ with $\# E\leq p^{d}/2$,
\begin{equation}
\sqrt{\# E/2}\leq \Phi(E)\leq \# E.
\end{equation}
See Babai \cite[Proposition 2.6]{Babai} for more details. The notation $\# E$ stands for the cardinality of a set $E$. Observe  that the optimal decay of $\E$ for all $\xi \neq 0$ are controlled by $O(\sqrt{\# E})$. We write $X=O( Y)$ means that  there is a positive constant $C$ such that $X\leq CY$, and $X=\Theta(Y)$ if $X=O(Y)$ and $Y=O(X)$. Isoevich and Rudnev \cite{IsoevichRudnev} called these sets Salem sets. To be precise we show the definition here.
\begin{definition}\label{def:Salem sets}\cite{IsoevichRudnev} A subset $E\subset \F$ is called a Salem set if for all non-zero $\xi$ of $\F$,
\begin{equation}
\m\widehat{E}(\xi)\m=O(\sqrt{\# E}).
\end{equation}
\end{definition}

Note that this is a finite fields version of Salem sets in Euclidean spaces. Roughly speaking, a set in Euclidean space is called a Salem set if there exist measures on this set, and  the Fourier transform of these measures have  optimal decay, see \cite{Bluhm},  \cite[Chapter 3]{Mattila} for more details on Salem sets in Euclidean spaces. 

It is well known that the sets for which all the non-zero Fourier coefficient  are small play an important role, e.g.,  see \cite{Babai}, \cite{Mattila} and \cite{TaoVu}. For some applications of  Salem sets in vector spaces over finite fields, see \cite{Iosevich}, \cite{IsoevichRudnev}, 
 \cite{Koh}.




In \cite[Theorem 1.13]{Hayes} Hayes proved  that almost all $m$-subset of $\F$ are (weak) Salem sets which answer a question of Babai. To be precise, let $E=E^{\omega}$ be selected uniformly at random from the collection of all subsets of $\F$ which have $m$ vectors. Let $\Omega(\F, m)$ denotes the probability space.

\begin{theorem}[Hayes] \label{thm:Hayes} Let $\varepsilon>0$. Let $m\leq p^{d}/2$. For all but an $O(p^{-d\varepsilon})$ probability  $E\in \Omega(\F, m)$, 
\begin{equation}
\Phi(E)< 2\sqrt{2(1+\varepsilon)m\log p^{d}}=O \left(\sqrt{m \log p^{d}} \right).
\end{equation}
\end{theorem}

For convenience we call this kink of subset of $\F$ weak Salem set. 






\subsection{Percolation on $\FF_{p}^{d}$ }

There is a another random model which is closely related to the random model $\Omega(\F,m)$. First we show this random model in the following. Let $0<\delta<1$. We choose each point of $\F$ with probability $\delta$ and remove it with probability $1-\delta$, all choices being independent of each other. Let $E=E^{\omega}$ be the collection of these chosen points, and $\Omega=\Omega (\F, \delta)$ be the probability space.  Note that both random models $\Omega(\F, m)$ and $\Omega(\F, \delta)$  
are related to the well known Erd\"os-R\'enyi-Gilbert random graph models.

We note that Hayes \cite{Hayes} proved a  similar result to Theorem \ref{thm:Hayes} for the random model $\Omega(\F, 1/2)$. However, the martingale  argument for $\Omega(\F, 1/2)$ and $\Omega(\F,m)$ of \cite{Hayes} do not  apply easily to the random model $\Omega(\F, \delta)$ for other values of $\delta\neq 1/2$. Babai \cite[Theorem 5.2]{Babai} used the Chernoff bounds for the model $\Omega(\F, 1/2)$, but it seems that the method also can not be easily extended to general $\delta$. We note that Babai \cite{Babai}, Hayes \cite{Hayes} proved their results in general finite Abelian group, see  \cite{Babai}, \cite{Hayes} for more details. For the finite vector space $\F$ (special Abel group) we extend their result to general $\delta$. 

\begin{theorem}\label{thm:maintheorem}
Let $\varepsilon>0$. Let $\delta\in (0,1)$. For all but an $O(p^{-d\varepsilon})$ probability $E \in  \Omega(\F, \delta)$,
\begin{equation}
\Phi (E)<2\sqrt{(1+\varepsilon)\delta p^{d}\log p^{d}}=O\left(\sqrt{\delta p^{d} \log p^{d}} \right).
\end{equation}
\end{theorem}

We know that almost all set $E\in \Omega(\F, \delta)$ has size roughly $\delta p^{d}$. This follows by  Chebyshev's inequality,  
\begin{equation}\label{eq:ee}
\PP(|\# E - p^{d}\delta|\geq  \frac{1}{2}p^{d}\delta)\leq \frac{4p^{d}\delta(1-\delta)}{(p^{d}\delta)^{2}}=O\left(\frac{1}{\delta p^{d}}\right).
\end{equation}
 
We immediately  have the following corollary, which says that almost all $E\in \Omega(\F, \delta)$ is a weak Salem set. 

\begin{corollary}
Let $\varepsilon>0$. Let $\delta\in (0,1)$. For all but an $O(\max\{p^{-d\varepsilon}, \frac{1}{\delta p^{d}}\})$ probability $E \in \Omega(\F,\delta)$,
\begin{equation}
|\widehat{E}(\xi)| = O\left(\sqrt{\#E \log p^{d}}\right). 
\end{equation}
\end{corollary}


In $\F$, it seems that  the only known examples of Salem sets are  discrete paraboloid and  discrete sphere. We note that both the size of the  discrete paraboloid and the discrete sphere are roughly $p^{d-1}$,  see \cite{IsoevichRudnev} for more details.  It is natural to ask that does there exists Salem set with any given size $m\leq p^{n}$. The above results and \cite[Problem 20]{Mattila2004} suggest  the following conjecture.

\begin{conjecture}
Let $s\in (0,d)$ be a non-integer and $C$ be a positive constant. Then  
\[
\min_{E}\frac{\Phi (E)}{\sqrt{\# E}} \rightarrow \infty \text{ as } p \rightarrow \infty,
\]  
where the minimal taking over all subsets $E \subset  \F$ with $p^{s}/C\leq \# E \leq   C p^{s}$. 
\end{conjecture}




\section{Preliminaries}\label{sec:pre}

In this section we show the definition of the finite field Fourier transform, and some easy facts about the random model $\Omega(\F, \delta)$.  Let $f : \F\longrightarrow \mathbb{C}$ be a complex value function. Then for $\xi \in \F$ we define the Fourier transform 
\begin{equation}
\widehat{f}(\xi)=\sum_{x\in \F} f(x)\e,
\end{equation}  
where the inter product $x \cdot \xi$ is defined as $ x_1\xi_1+\cdots +x_p\xi_p$. 
Recall the following Plancherel identity, 
\begin{equation*}
 \sum_{\xi \in \F}|\widehat{f}(\xi)|^{2}=p^{d}\sum_{x\in \F} |f(x)|^{2}.
\end{equation*} 
Specially for the subset of $E\subset \F$, we have 
\begin{equation}
\sum_{\xi \in \F} \m\E\m^{2}=p^{d}\# E.
\end{equation}
For more details on discrete Fourier analysis, see  Stein and Shakarchi \cite{Stein}.

We show some easy facts about the random model $\Omega(\F, \delta)$ in the following. Let $\xi\neq 0$, then the expectation of $\E$ is  
\begin{equation*}
\EE(\widehat{E}(\xi))=\delta \sum_{x \in \F}\e=0.
\end{equation*}
Since 
\begin{equation*}
\begin{aligned}
\m \widehat{E}(\xi)\m^{2}&=\sum_{x,y \in \F}E(x)E(y)e^{-\frac{2 \pi i (x-y)\cdot \xi}{p}}\\
&=\sum_{x \in \F}E(x)+\sum_{x\neq y \in \F}E(x)E(y)e^{-\frac{2 \pi i (x-y)\cdot \xi}{p}},
\end{aligned}
\end{equation*}
we have 
\begin{equation*}
\begin{aligned}
\EE \left(\m \widehat{E}(\xi)\m^{2} \right)&= \delta p^{d}+ \delta^{2}\sum_{x\neq y \in \F}e^{-\frac{2 \pi i (x-y)\cdot \xi}{p}}\\
&=p^{d}\delta \left(1-\delta \right).
\end{aligned}
\end{equation*}
We may read this identity as (for small $\delta$) 
\[
\m\E\m =\Theta \left( \sqrt{p^{d}\delta} \right) = \Theta \left( \sqrt{\# E} \right).  
\]

\section{Proof of Theorem \ref{thm:maintheorem}}

For the convenience to our use, we formulate a special large deviations estimate in the following. For more background and details on large deviations estimate, see Alon and Spencer \cite[Appendix A]{Alon}.

\begin{lemma}\label{lem:law of large numbers}
Let $\{X_j\}_{j=1}^N$ be a sequence  independent random variables with $ \vert X_j \vert \leq 1$,  $\mu_1:= \sum_{j=1}^{N}\EE(X_i)$, and $\mu_{2}:=\sum_{j=1}^{N}\EE(X_j^{2})$. Then for any $\alpha>0$, $0<\lambda<1$, 
\begin{equation}
\PP( \big\m \sum^N_{j=1} X_j \big \m \geq  \alpha )\leq e^{-\lambda \alpha+\lambda^{2}\mu_{2}}(e^{\lambda \mu _{1}}+e^{-\lambda \mu _{1}}).
\end{equation} 
\end{lemma}
\begin{proof}
Applying Markov's inequality to the random variable $e^{\lambda \sum^N_{j=1} X_j}$. This gives 
\begin{equation}\label{eq:markov}
\begin{aligned}
\PP( \sum^N_{j=1} X_j \geq \alpha)&=\PP (e^{\lambda \sum^N_{j=1} X_j}  > e^{\lambda \alpha})\\
& \leq  e^{-\lambda \alpha} \EE(e^{\lambda \sum^N_{j=1} X_j})\\
&=e^{\lambda \alpha} \prod^N_{j=1} \EE( e^{ \lambda X_j}),
\end{aligned}
\end{equation}
the last equality holds since $\{X_j\}_j$ is a sequence independent random variables. 

For any  $\m x \m \leq 1$ we have  
 \[
 e^{x} \leq 1+x +x^{2}.
 \] 
Since $\m \lambda  X_j \m \leq 1$, we have 
 \[
 e^{\lambda X_j}\leq 1 + \lambda X_j +\lambda^{2} X_j^{2}, 
 \]
 and hence 
\begin{equation*}
\begin{aligned}
\EE( e^{\lambda X_j}) &\leq 1+ \EE(\lambda X_i)+\EE (\lambda^{2} X_j^{2})\\
&\leq e^{\EE(\lambda X_i)+\EE (\lambda^{2} X_j^{2})}.
\end{aligned}
\end{equation*}
Combining this with \eqref{eq:markov}, we have
\[
\PP( \sum^N_{j=1} X_j \geq \alpha)\leq e^{-\lambda \alpha + \lambda \mu _{1} + \lambda^{2}\mu _{2}}.
\]

Applying the similar way to the above for  $\PP(\sum_{j=1}^{N} X_{j} \geq -\alpha)$, we obtain
\[
\PP( -\sum^N_{j=1} X_j \geq \alpha)\leq e^{-\lambda \alpha - \lambda \mu _{1} + \lambda^{2}\mu _{2}}.
\]
Thus we finish the proof.
\end{proof}

The following two easy identities are also useful for us. 
\begin{equation}\label{eq:identities}
\begin{aligned}
&\sum_{x\in \F} \cos\frac{2 \pi x \cdot\xi}{p}=Re \left(\sum_{x\in \F}\e \right)=0\\
&\sum_{x\in \F} \cos^{2}\frac{2 \pi x \cdot\xi}{p}=\sum_{x\in \F}\frac{1+\cos\frac{4 \pi x \cdot\xi}{p}}{2} =\frac{1}{2}p^{d}
\end{aligned}
\end{equation}

\begin{proof}[Proof of Theorem \ref{thm:maintheorem}]
Let $\xi\neq 0$ and $E\in \Omega(\F, \delta)$.  Let 
\[
\widehat{E}(\xi)=\sum_{x\in \F} E(x)\e =\mathcal{R} +i \mathcal{I}
\] 
where $\R$ and is  the real part  of $\widehat{E}(\xi)$,  and $\I$ is the imagine part of $\widehat{E}(\xi)$. First we provide the estimate to the real part $\R$. By the Euler identity, we have
\[
\mathcal{R}=\sum_{x\in \F}E(x)\cos(\frac{2\pi x\cdot \xi}{p}).
\]
Note that 
\[
E(x)\cos\left(\frac{2\pi x\cdot \xi}{p}\right), x\in \F
\] is a sequence of independent random variables. Furthermore, applying the identities \eqref{eq:identities} , we have  
 
\begin{equation}\label{eq:mu}
\mu _{1}=0, \,\, \mu _{2}=\frac{1}{2}p^{d}\delta.
\end{equation}
Here $\mu _{1}, \mu _{2}$ are defined as the same way as in the Lemma \ref{lem:law of large numbers}. Let 
\begin{equation}\label{eq:alpha}
\alpha :=\sqrt{2(1+\varepsilon) p^{d}\delta \log p^{d}},\,  \, \lambda:=\frac{\alpha}{p^{d}\delta}.
\end{equation}

Note that $\lambda\leq 1$  for large $p$. Applying Lemma \ref{lem:law of large numbers}, we have 
\begin{equation}\label{eq:r}
\begin{aligned}
\PP(\m\R\m \geq  \alpha)&\leq 2 e^{-\lambda \alpha +\lambda^{2} \mu _{2}}\\
&=2e^{-\frac{\alpha^{2}}{2p^{d}\delta}}=\frac{2}{p^{d(1+\varepsilon)}}.
\end{aligned}
\end{equation}

Now we turn to the  imagine part $\I$. Applying the similar argument to the real part $\R$, note that the identities \eqref{eq:identities} also hold if we take $\sin$ instead of $\cos$, we obtain 
\[
\PP(|\I | \geq \alpha)\leq \frac{2}{p^{d(1+\varepsilon)}}.
\]
Combining this with the estimate \eqref{eq:r}, we obtain
\begin{equation}
\PP( |\widehat{E}(\xi)| \geq  \sqrt{2}\alpha)\leq \PP(\m\R\m \geq  \alpha)+\PP(|\I | \geq \alpha) \leq  \frac{4}{p^{d(1+\varepsilon)}}
\end{equation}

Observe that the above argument  works to any non-zero vector $\xi$. Therefore, we obtain 
\begin{equation}
\PP( \exists \, \xi \neq 0, \text { s.t }  |\widehat{E}(\xi)| \geq \sqrt{2}\alpha)\leq  \frac{4}{p^{d\varepsilon}}.
\end{equation}
Recall the value of $\alpha$ in \eqref{eq:alpha}, 
\[
\alpha =\sqrt{2(1+\varepsilon) p^{d}\delta \log p^{d}},
\]
this completes the proof.
\end{proof}

\medskip
\textbf{Acknowledgements.} I am grateful  for being  supported by the Vilho, Yrj\"o, and Kalle V\"ais\"al\"a foundation.


 





\end{document}